\documentclass{amsart}

\usepackage[T1]{fontenc}
\usepackage{amssymb,mathtools}
\usepackage[alphabetic,nobysame]{amsrefs}
\usepackage{hyperref}

\theoremstyle{definition}
\newtheorem{defi}{Definition}[section]
\newtheorem{rem}[defi]{Remark}
\theoremstyle{plain}
\newtheorem{thm}[defi]{Theorem}
\newtheorem{prop}[defi]{Proposition}
\newtheorem{lem}[defi]{Lemma}
\newtheorem{cor}[defi]{Corollary}

\newcommand{\spec}{\operatorname{Spec}}
\newcommand{\supp}{\operatorname{Supp}}

\newcommand{\mdepth}{\operatorname{depth}}
\newcommand{\RGamma}{\mathbf{R}\Gamma}

\newcommand{\ideal}[1]{\mathfrak{#1}}

\renewcommand{\hom}{\operatorname{Hom}}

\newcommand{\ajRoman}[1]{\MakeUppercase{\romannumeral #1}}

\makeatletter
\newcommand{\ilim}[1][\relax]{%
  \def\ilim@chk{#1}%
  \@xp\ifx\ilim@chk\relax
    \varinjlim
  \else
    \mathop{\varinjlim}\limits_{#1}%
  \fi}
\newcommand{\plim}[1][\relax]{%
  \def\plim@chk{#1}%
  \@xp\ifx\plim@chk\relax
    \varprojlim
  \else
    \mathop{\varprojlim}\limits_{#1}%
  \fi}
\makeatother

\title{Cohomological Cohen--Macaulayness in Non-Noetherian Rings}
\author{Ryoya Ando}
\address{Beyont Ltd., Chiyoda, Japan.}
\email{r\_ando@skillupai.com}
\email{6122701@alumni.tus.ac.jp}
\date{}
\subjclass[2020]{13H10 (Primary), 13D45, 13A35 (Secondary).}
\keywords{non-Noetherian Cohen--Macaulay ring, cohomology with support, big Cohen--Macaulay algebra}

\begin{document}

\begin{abstract}
  We study Cohen--Macaulayness, in the sense of Hamilton--Marley, of non-Noetherian rings arising as big
  Cohen--Macaulay algebras. Motivated by Bhatt's notion of cohomological Cohen--Macaulayness, we call a
  locally finite-dimensional ring CCM if its structure sheaf satisfies this condition. Our main comparison
  theorem shows that every locally finite-dimensional CCM ring is locally HMCM\@. Using this theorem, we prove
  that if $A$ is Noetherian and $R$ is an integral $A$-algebra that is locally balanced big Cohen--Macaulay
  over $A$, then $R$ is CCM and hence locally HMCM\@. In particular, if $A$ is an excellent Noetherian domain,
  $p$ is a prime, $n\geq1$, and $A/pA\neq0$, then $A^+/p^nA^+$ is CCM and locally HMCM\@. We also show, using
  finite-dimensional valuation domains, that CCM is strictly stronger than locally HMCM\@. Finally, for a
  Noetherian ring $A$ of characteristic $p>0$ and its perfection $A_{\mathrm{perf}}$, we prove that the
  following conditions are equivalent: $A$ is locally weakly $F$-nilpotent; $A_{\mathrm{perf}}$ is a locally
  balanced big Cohen--Macaulay $A$-algebra; and $A_{\mathrm{perf}}$ is CCM\@. Under these equivalent
  conditions, $A_{\mathrm{perf}}$ is locally HMCM\@.
\end{abstract}

\maketitle

\section{Introduction}

For a Noetherian local ring, depth, regularity of systems of parameters, and the vanishing of lower local
cohomology are equivalent characterizations of Cohen--Macaulayness. Once the Noetherian hypothesis is
removed, however, regular sequences, Koszul and \v{C}ech cohomology, module-theoretic local cohomology, and
sheaf-theoretic cohomology with support may behave differently, and these characterizations are no longer
equivalent in general. Accordingly, several notions of Cohen--Macaulayness for non-Noetherian rings have been
proposed by extending different classical characterizations. Hamilton--Marley introduced a definition based
on parameter sequences and \v{C}ech cohomology \cite{Hamilton-Marley2007}, while Asgharzadeh--Tousi compared
several non-Noetherian Cohen--Macaulay conditions, including those based on Koszul grade and height and weak
Bourbaki unmixedness \cite{Asgharzadeh-Tousi2009}. Throughout this paper, we call the Hamilton--Marley notion
of Cohen--Macaulayness (Definition~\ref{def:HMCM}) HMCM\@.

One motivation for this work comes from big Cohen--Macaulay algebras. Big Cohen--Macaulayness is defined
relative to a Noetherian base ring, while such algebras are often non-Noetherian. It is natural to ask
whether such an algebra satisfies a non-Noetherian Cohen--Macaulay condition as a ring. If $A$ is an
excellent Noetherian domain of characteristic $p>0$, Hochster--Huneke proved that $A^+$ is a balanced big
Cohen--Macaulay $A$-algebra \cite{Hochster-Huneke1992}. Hamilton--Marley further showed that $A^+$ is HMCM
\cite{Hamilton-Marley2007}*{Theorem~4.11}, and Asgharzadeh--Tousi established several other non-Noetherian
Cohen--Macaulay properties for absolute integral closures in positive characteristic
\cite{Asgharzadeh-Tousi2009}*{Theorem~4.4}. For this purpose, we use the following cohomological condition.

Bhatt \cite{Bhatt2021}*{Definition~2.1} introduced cohomological Cohen--Macaulayness for complexes on
finite-dimensional schemes in terms of cohomology with support at points. The same local vanishing condition
appears for modules over Noetherian local rings in Bhatt--Ma--Patakfalvi--Schwede--Tucker--Waldron--Witaszek
\cite{BMPSTWW2023}*{Discussion~2.2}, and for topologically Noetherian schemes in Cass--Louren\c{c}o
\cite{Cass-Lourenco2025}*{Definition~2.9}. For a ring $A$ whose local rings are all finite-dimensional, we
say that $A$ is \emph{cohomologically Cohen--Macaulay in the sense of Bhatt}, abbreviated CCM, if for every
$P\in\spec A$,
\[
  H^{i}_{PA_P}(\spec A_P,\mathcal O)=0\qquad(i<\dim A_P).
\]
This condition is local by definition and agrees with the usual Cohen--Macaulay property for Noetherian
rings.

We first prove the following comparison theorem.

\begin{thm}[Theorem A]\label{thm:intro-main}
  Let $A$ be a locally finite-dimensional CCM ring. Then $A$ is locally HMCM\@. More precisely, for
  every $P\in\spec A$ and every parameter sequence $\underline{a}$ of length $r$ in $A_P$, we have $\check{H}^i(\underline{a},A_P)=0$ for every $i<r$.
\end{thm}

The key point of the proof is to choose a minimal point in the support of a nonzero lower \v{C}ech cohomology
module and use a Grothendieck spectral sequence together with the Hamilton--Marley height inequality.

We next apply this comparison theorem to big Cohen--Macaulay algebras. For a Noetherian ring $A$, an
$A$-module $M$ is called \emph{locally balanced big Cohen--Macaulay} if $M_P$ is balanced big Cohen--Macaulay
over $A_P$ for every $P\in\spec A$.

\begin{thm}[Theorem B]\label{thm:intro-transfer}
  Let $A$ be a Noetherian ring, and let $R$ be an integral $A$-algebra that is locally balanced big
  Cohen--Macaulay over $A$. Then $R$ is a locally finite-dimensional CCM ring and hence is locally
  HMCM\@.
\end{thm}

Combining Theorems A and B, we obtain
\[
  \text{locally balanced big Cohen--Macaulayness over }A
  ~\Longrightarrow~
  \mathrm{CCM}
  ~\Longrightarrow~
  \text{locally HMCM}.
\]
The first implication is Theorem B, and the second is Theorem A.

Applying this to absolute integral closures gives the following.

\begin{cor}[Theorem C]\label{cor:intro-absolute}
  Let $A$ be an excellent Noetherian domain, let $p$ be a prime, and let $n\geq1$. Assume that
  $A/pA\neq0$. Then $A^+/p^nA^+$ is CCM and locally HMCM\@. In particular, if
  $\operatorname{char}A=p>0$, then $A^+$ is CCM and locally HMCM\@.
\end{cor}

In positive characteristic, the HMCM property of $A^+$ was proved by Hamilton--Marley
\cite{Hamilton-Marley2007}*{Theorem~4.11}. Bhatt \cite{Bhatt2021}*{Theorem~1.1} proved that $A^+/p^nA^+$ is
Cohen--Macaulay over $A/p^nA$. Applying Theorem B to these big Cohen--Macaulay algebras shows that
$A^+/p^nA^+$ is CCM and locally HMCM. Since HMCM is not preserved by localization in general \cite{Ando2026},
even in positive characteristic local HMCM does not follow from the HMCM property alone.

Nilpotence of the Frobenius action on local cohomology has been studied in connection with $F$-singularities.
Srinivas--Takagi studied $F$-nilpotent singularities and their relation to the Hodge filtration in
characteristic zero \cite{Srinivas-Takagi2017}. Polstra--Quy introduced relative Frobenius actions and
characterized $F$-nilpotence, as well as nilpotence of Frobenius on lower local cohomology, in terms of
Frobenius closure \cite{Polstra-Quy2019}. Quy further proved that under the latter condition the Frobenius
test exponents of parameter ideals are uniformly bounded \cite{Quy2019}. Here we focus on the latter
condition, namely weak $F$-nilpotence.

For a ring $A$ of characteristic $p>0$, write
\[
  A_{\mathrm{perf}}:=\varinjlim\bigl(A\xrightarrow{F}A\xrightarrow{F}\cdots\bigr)
\]
for its perfection. Using the local characterization of Ma--Polstra together with the local-cohomology
computation in Section~\ref{sec:perfect-closure}, we obtain the following characterization.

\begin{thm}[Theorem D]\label{thm:intro-perfect}
  Let $A$ be a Noetherian ring of characteristic $p>0$. The following are equivalent.
  \begin{enumerate}
    \item $A$ is locally weakly $F$-nilpotent.
    \item $A_{\mathrm{perf}}$ is a locally balanced big Cohen--Macaulay $A$-algebra.
    \item $A_{\mathrm{perf}}$ is CCM\@.
  \end{enumerate}
  If these conditions hold, then $A_{\mathrm{perf}}$ is locally HMCM\@.
\end{thm}

For local rings, the equivalence \textup{(1)}$\Longleftrightarrow$\textup{(2)} follows from Ma--Polstra
\cite{Ma-Polstra2025}*{Proposition~12.22}. The implication \textup{(2)}$\Longrightarrow$\textup{(3)} follows
from Theorem B, while the converse implication is proved in Section~\ref{sec:perfect-closure}. The final
assertion follows from Theorem A. For a related computation of local cohomology of the perfection and the
Frobenius action under additional hypotheses, see also Cass--Louren\c{c}o
\cite{Cass-Lourenco2025}*{Lemma~2.10}.

The paper is organized as follows. Section~\ref{sec:preliminaries} recalls cohomology with support, HMCM, and
CCM\@. Section~\ref{sec:comparison} proves Theorem A. Section~\ref{sec:integral-applications} proves Theorems
B and C, and Section~\ref{sec:perfect-closure} proves Theorem D.

\section{Two notions of Cohen--Macaulayness for non-Noetherian rings}\label{sec:preliminaries}
\subsection{Hamilton--Marley Cohen--Macaulayness}

Let $A$ be a ring and $M$ an $A$-module. For an ideal $I\subseteq A$, let $\Gamma_I(-)\coloneq \ilim_n
  \hom_A(A/I^n,-)$ be the usual left exact functor, and write $H^i_I(M)$ for its right derived functors, the
local cohomology modules of $M$ with support in $I$. For a finite sequence $\underline{a}=a_1,\dots,a_r$, we
also consider the \v{C}ech cohomology $\check{H}^i(\underline{a},M)$. A finite sequence $\underline{a}$ is
\emph{weakly proregular} if and only if local cohomology and \v{C}ech cohomology agree functorially for every
$A$-module $M$ \cite{Schenzel2003}*{Theorem~3.2}. Over a Noetherian ring, every finite sequence is weakly
proregular.

\begin{defi}[Hamilton--Marley]\label{def:HMCM}
  A finite sequence $\underline{a}=a_1,\dots,a_r$ in a ring $A$ is called a
  \emph{parameter sequence} if the following conditions hold:
  \begin{enumerate}
    \item $\underline{a}$ is weakly proregular;
    \item $\underline{a}A\neq A$;
    \item for every $P\supseteq\underline{a}A$, $\check{H}^r(\underline{a},A)_P\neq0$.
  \end{enumerate}
  If every initial subsequence of $\underline{a}$ is a parameter sequence, then
  $\underline{a}$ is called a \emph{strong parameter sequence}.

  The ring $A$ is called \emph{HMCM} if every strong parameter sequence is an $A$-regular sequence. We say that
  $A$ is \emph{locally HMCM} if $A_P$ is HMCM for every $P\in\spec A$.
\end{defi}

By \cite{Hamilton-Marley2007}*{Proposition~4.2}, a ring $A$ is HMCM if and only if for every strong parameter
sequence $\underline{a}$ of length $r$, $\check{H}^i(\underline{a},A)=0 ~ (i<r)$. Moreover, every coherent
regular ring is locally HMCM \cite{Hamilton-Marley2007}*{Theorem~4.8}.

\subsection{Cohomological Cohen--Macaulayness in the sense of Bhatt}

We recall sheaf-theoretic local cohomology. Let $X$ be a scheme, let $Z\subseteq X$ be closed, and let
$\mathcal F$ be an $\mathcal O_X$-module. We write $\RGamma_Z(X,\mathcal F)$ for the right derived functor of
sections with support in $Z$, and $H_Z^i(X,\mathcal F)$ for its cohomology. In general, this need not agree
with either module-theoretic local cohomology or \v{C}ech cohomology.

\begin{prop}\label{prop:sheaf-cech}
  Let $A$ be a ring, let $M$ be an $A$-module, let
  $\underline{a}=a_1,\dots,a_r$, and set $I=(a_1,\dots,a_r)$. Then, for every $i\geq0$, there is a
  natural isomorphism
  \[
    H^i_{V(I)}(\spec A,\widetilde M)\cong\check{H}^i(\underline{a},M).
  \]
\end{prop}

\begin{proof}
  This follows from \cite{Grothendieck1967}*{Theorem~2.3 and Proposition~C in its proof}.
\end{proof}

Consequently, if $I$ is generated by a weakly proregular finite sequence $\underline{a}$, then
$H^i_{V(I)}(\spec A,\widetilde M)\cong H^i_I(M)$.

\begin{defi}\label{def:bhatt-ccm}
  A ring $A$ is called \emph{locally finite-dimensional} if $\dim A_P<\infty$ for every
  $P\in\spec A$. Such a ring $A$ is called \emph{cohomologically Cohen--Macaulay in the sense of
    Bhatt} if, for every $P\in\spec A$,
  \[
    \mathbf{R}\Gamma_{PA_P}(\spec A_P,\mathcal{O})\in D^{\geq\dim A_P},
  \]
  equivalently,
  \[
    H^i_{PA_P}(\spec A_P,\mathcal{O})=0 \qquad \text{for } i<\dim A_P.
  \]
  We abbreviate this condition by saying that $A$ is \emph{CCM}.
\end{defi}

This condition is inspired by the notion of a cohomologically Cohen--Macaulay complex in
\cite{Bhatt2021}*{Definition~2.1}, specialized to $X=\spec A$ and the structure sheaf $\mathcal O_X$. Note
also that every Noetherian ring is locally finite-dimensional by Krull's height theorem.

The following is immediate from the definition.

\begin{prop}\label{prop:ccm-local}
  The CCM property is preserved under localization. In particular, $A$ is CCM if and only if
  $A_{\ideal{m}}$ is CCM for every maximal ideal $\ideal{m}$ of $A$.
\end{prop}

We also record compatibility with the classical Noetherian notion.

\begin{prop}\label{prop:noetherian-compatibility}
  A Noetherian ring $A$ is CCM if and only if it is Cohen--Macaulay in the usual sense.
\end{prop}

\begin{proof}
  For a Noetherian ring, sheaf-theoretic local cohomology agrees with the usual local cohomology.
  Moreover, by the local cohomology characterization of depth
  \cite{Bruns-Herzog1997}*{Theorem~3.5.7},
  \[
    \mdepth A_P=\inf\{i\mid H^i_{PA_P}(A_P)\neq0\}.
  \]
  Thus the vanishing of local cohomology for $i<\dim A_P$ is equivalent to $\mdepth A_P=\dim A_P$.
\end{proof}

\section{Comparison between CCM and locally HMCM}\label{sec:comparison}

\begin{thm}\label{thm:ccm-to-locally-hmcm}
  Let $A$ be a locally finite-dimensional CCM ring. Then $A$ is locally HMCM\@. More precisely, for every $P\in\spec A$ and every parameter sequence $\underline{a}$ of length $r$ in $A_P$, we have $\check{H}^i(\underline{a},A_P)=0$ for every $i<r$.
\end{thm}

\begin{proof}
  By Proposition~\ref{prop:ccm-local}, every localization of a CCM ring is a finite-dimensional CCM local ring. Thus it suffices to show that every finite-dimensional CCM local ring is HMCM\@.

  Let $\underline{a}=a_1,\dots,a_r$ be a parameter sequence of $A$, and set $i=\min\{j\leq r\mid
    \check{H}^j(\underline{a},A)\neq0\}$. Since $\supp\check{H}^i(\underline{a},A)$ is nonempty and $\dim
    A<\infty$, it has a minimal element $P$. Moreover, $\supp\check{H}^i(\underline{a},A)\subseteq
    V(\underline{a}A)$. Indeed, let $Q\notin V(\underline{a}A)$. Then $a_j\notin Q$ for some $j$. Since
  $\check{H}^i(\underline{a},A)_Q=\check{H}^i(\underline{a},A_Q)$ and $a_j$ is a unit in $A_Q$, we have
  $\check{H}^i(\underline{a},A_Q)=0$; see \cite{StacksProject}*{Tag~0G6J}.

  Set $X=\spec A_P$ and $Z=V(\underline{a}A_P)$. By Proposition~\ref{prop:sheaf-cech}, for every $j$ we have
  $\check{H}^j(\underline{a},A)_P\cong H^j_Z(X,\mathcal{O}_X)$. For the local cohomology sheaf
  $\mathcal{H}^j_Z(\mathcal{O}_X)$, \cite{SGA2}*{EXP. \ajRoman{2}, Cor.~4 and Prop.~5} gives
  \[
    \mathcal{H}^j_Z(\mathcal{O}_X)
    =\widetilde{\check{H}^j(\underline{a},A)_P}.
  \]
  Since $PA_P\in Z$, sections supported at $PA_P$ are already supported in $Z$, so
  \[
    \Gamma_{PA_P}(X,\mathcal{H}^0_Z(-))=\Gamma_{PA_P}(X,-).
  \]
  Hence, by \cite{Grothendieck1957}*{Theorem~2.4.1}, there is a spectral sequence
  \[
    E_2^{p,q}
    =H^p_{PA_P}(X,\mathcal{H}^q_Z(\mathcal{O}_X))
    \Rightarrow H^{p+q}_{PA_P}(X,\mathcal{O}_X).
  \]

  Let $U\coloneq X\setminus\{PA_P\}$. For $q=i$, \cite{Grothendieck1967}*{Cor.~1.9} gives, for
  $N\coloneq\check{H}^i(\underline{a},A)_P\neq0$, the long exact sequence
  \[
    \begin{aligned}
      0 &\rightarrow \Gamma_{PA_P}(X,\widetilde{N})
      \rightarrow \Gamma(X,\widetilde{N})
      \rightarrow \Gamma(U,\widetilde{N}) \\
      &\rightarrow H^1_{PA_P}(X,\widetilde{N})
      \rightarrow H^1(X,\widetilde{N})
      \rightarrow H^1(U,\widetilde{N})
      \rightarrow \cdots.
    \end{aligned}
  \]
  By the minimality of $P$, the sheaf $\widetilde{N}$ is supported only at $\{PA_P\}$. Hence
  $H^p_{PA_P}(X,\widetilde{N})\cong H^p(X,\widetilde{N})$ for every $p$. Since $X=\spec A_P$ is affine and
  $\widetilde{N}$ is quasi-coherent, $H^0_{PA_P}(X,\widetilde{N})=N$ and $H^p_{PA_P}(X,\widetilde{N})=0$ for
  $p>0$. Thus $E_2^{0,i}=N$ and $E_2^{p,i}=0$ for $p>0$. Moreover, $E_2^{p,j}=0$ for $j<i$, and therefore
  \[
    H^i_{PA_P}(X,\mathcal{O}_X)
    =E_{\infty}^{0,i}
    =E_2^{0,i}
    =N\neq0.
  \]
  Since $A$ is CCM, we must have $i\geq\dim A_P$. On the other hand, $\underline{a}A\subseteq P$, so
  \cite{Hamilton-Marley2007}*{Proposition~3.6} gives $i\leq r\leq\dim A_P$. Hence $i=r$. The characterization
  of Hamilton--Marley \cite{Hamilton-Marley2007}*{Proposition~4.2} now shows that $A$ is HMCM\@.
\end{proof}

\begin{rem}
  The locally finite-dimensional hypothesis is used in the proof above only to choose a minimal element in the support of a nonzero \v{C}ech cohomology module. Thus the same argument applies whenever, for every localization and every parameter sequence, the nonzero supports of the lower \v{C}ech cohomology modules have minimal elements.
\end{rem}

We next give examples of locally HMCM rings that are not CCM\@.

\begin{prop}
  Let $V$ be a valuation ring of finite Krull dimension. Then the following are equivalent:
  \begin{enumerate}
    \item $V$ is CCM\@.
    \item $\dim V\leq1$.
  \end{enumerate}
  Moreover, every valuation ring is a coherent regular ring and is locally HMCM\@. In particular, not every coherent regular ring or locally HMCM ring is CCM\@.
\end{prop}

\begin{proof}
  Assume that $\dim V\leq1$. For every $P\in\spec V$, we have $\dim V_P\leq1$. If $\dim V_P=0$, the CCM condition is vacuous. If $\dim V_P=1$, then, since $V_P$ is a domain,
  \[
    H^0_{PV_P}(\spec V_P,\mathcal{O})
    =\Gamma_{PV_P}(V_P)
    =0.
  \]
  Hence $V$ is CCM\@.

  Conversely, assume that $\dim V\geq2$. Let $\ideal{m}$ be the maximal ideal of $V$ and set $U=\spec
    V\setminus\{\ideal{m}\}$. Since $V$ has finite dimension, in particular $\#\spec V<\infty$, and hence $U$ is
  quasi-compact. By \cite{Datta2017}*{Lemma~6.2}, there exists $f\in\ideal{m}$ such that $U=D(f)=\spec V_f$.
  Furthermore, \cite{Datta2017}*{Theorem~6.1} gives
  \[
    H^1_{\ideal{m}}(\spec V,\mathcal{O})
    \cong
    \operatorname{coker}(V\rightarrow V_f)
    =V_f/V.
  \]
  Since $f$ is a nonunit, $1/f\notin V$, and hence $V_f/V\neq0$. Since $1<\dim V$, it follows that $V$ is not
  CCM\@.

  Finally, every finitely generated ideal of a valuation ring is principal and hence projective, so $V$ is a
  coherent regular ring. By \cite{Hamilton-Marley2007}*{Theorem~4.8}, $V$ is locally HMCM\@.
\end{proof}

\section{Applications to integral big Cohen--Macaulay algebras}\label{sec:integral-applications}

\begin{defi}
  Let $A$ be a Noetherian ring and $M$ an $A$-module. We say that $M$ is
  \emph{locally balanced big Cohen--Macaulay} if, for every $P\in\spec A$, the localization $M_P$ is a balanced big Cohen--Macaulay $A_P$-module; that is,
  $M_P/P M_P\neq0$, and every system of parameters of $A_P$ is an
  $M_P$-regular sequence.
\end{defi}

If an $A$-algebra $R$ satisfies this condition as an $A$-module, we call $R$ a locally balanced big
Cohen--Macaulay $A$-algebra.

\begin{thm}\label{thm:integral-transfer}
  Let $A$ be a Noetherian ring and let $R$ be a locally balanced big Cohen--Macaulay $A$-algebra that is integral over $A$.
  Then $R$ is a locally finite-dimensional CCM ring, and hence is locally HMCM\@.
\end{thm}

\begin{proof}
  Let $\mathfrak{p}\in\spec R$ and set $P=\mathfrak{p}\cap A$.
  Since $R$ is a locally balanced big Cohen--Macaulay $A$-algebra and is integral over $A$, $R_P$ is a balanced big Cohen--Macaulay $A_P$-algebra and $A_P\to R_P$ is integral.
  Set $d=\dim A_P$. Since $A\to R$ is integral, we have
  $\dim R_{\mathfrak{p}}\leq d$.
  It therefore suffices to show that $H^i_{\mathfrak{p}R_{\mathfrak{p}}}(\spec R_{\mathfrak{p}},\mathcal{O})=0$ for every $i<d$.

  Let $\underline{a}=a_1,\dots,a_d\in A_P$ be a system of parameters of $A_P$. Since $R_P$ is a balanced big
  Cohen--Macaulay $A_P$-algebra, $\underline{a}$ is an $R_P$-regular sequence. Now
  $R_{\mathfrak{p}}=(R_P)_{\mathfrak{p}R_P}$, and regularity is preserved under localization, so
  $\underline{a}$ is also an $R_{\mathfrak{p}}$-regular sequence.

  We claim that $V(\underline{a}R_{\mathfrak{p}})=\{\mathfrak{p}R_{\mathfrak{p}}\}$. Let
  $\mathfrak{q}R_{\mathfrak{p}}\in V(\underline{a}R_{\mathfrak{p}})$. Its contraction to $R_P$ is
  $\mathfrak{q}R_P=\mathfrak{q}R_{\mathfrak{p}}\cap R_P$. Since
  $\underline{a}R_{\mathfrak{p}}\subset\mathfrak{q}R_{\mathfrak{p}}$, we have $\underline{a}A_P\subset
    \mathfrak{q}R_P\cap A_P\subset PA_P$. Since $\underline{a}$ is a system of parameters of $A_P$,
  $\sqrt{\underline{a}A_P}=PA_P$. Hence $\mathfrak{q}R_P\cap A_P=PA_P$. On the other hand, $\mathfrak{p}R_P\cap
    A_P=PA_P$. Since $A_P\to R_P$ is integral and $\mathfrak{q}R_P\subseteq\mathfrak{p}R_P$, incomparability
  gives $\mathfrak{q}R_P=\mathfrak{p}R_P$. Thus $\mathfrak{q}R_{\mathfrak{p}}=\mathfrak{p}R_{\mathfrak{p}}$,
  proving the claim.

  By Proposition~\ref{prop:sheaf-cech}, $H^i_{\mathfrak{p}R_{\mathfrak{p}}}(\spec
    R_{\mathfrak{p}},\mathcal{O})=\check{H}^i(\underline{a},R_{\mathfrak{p}})$. Since $\underline{a}$ is an
  $R_{\mathfrak{p}}$-regular sequence and \v{C}ech cohomology is a direct limit of Koszul cohomology, we have
  $\check{H}^i(\underline{a},R_{\mathfrak{p}})=0$ for every $i<d$.

  Thus $R$ is CCM\@. By Theorem~\ref{thm:ccm-to-locally-hmcm}, $R$ is locally HMCM\@.
\end{proof}

\begin{cor}\label{cor:absolute-positive-characteristic}
  Let $A$ be an excellent Noetherian domain of characteristic $p>0$.
  Then $A^+$ is CCM and locally HMCM\@.
\end{cor}

\begin{proof}
  Let $P\in\spec A$. Then $A_P$ is an excellent Noetherian local domain of characteristic $p>0$.
  By Hochster--Huneke \cite{Hochster-Huneke1992}, every system of parameters of $A_P$ is a regular sequence on $(A_P)^+$.
  Since absolute integral closure commutes with localization, $(A^+)_P\cong (A_P)^+$.
  Therefore $A^+$ is a locally balanced big Cohen--Macaulay $A$-algebra.
  Since $A\to A^+$ is integral, Theorem~\ref{thm:integral-transfer} shows that $A^+$ is CCM and locally HMCM\@.
\end{proof}

Applying Theorem~\ref{thm:integral-transfer} to Bhatt's theorem \cite{Bhatt2021} gives the following.

\begin{cor}[absolute integral closure modulo $p^n$]\label{cor:absolute-integral-closure}
  Let $A$ be an excellent Noetherian domain, let $p$ be a prime number, and let $n\geq1$.
  Assume that $A/pA\neq0$.
  Let $A^+$ denote the absolute integral closure of $A$.
  Then $A^+/p^nA^+$ is CCM and locally HMCM\@.
\end{cor}

\begin{proof}
  Set $B=A/p^nA$ and $R=A^+/p^nA^+$.
  The map $B\to R$ is integral, and by \cite{Bhatt2021}*{Theorem~1.1}, $R$ is a locally balanced big Cohen--Macaulay $B$-algebra.
  The conclusion follows from Theorem~\ref{thm:integral-transfer}.
\end{proof}

\section{Perfection and weak \texorpdfstring{$F$}{F}-nilpotence}\label{sec:perfect-closure}

In the previous section, we proved that a locally balanced big Cohen--Macaulay algebra integral over a
Noetherian ring is CCM, and hence locally HMCM\@. We now apply Theorem~\ref{thm:integral-transfer} to the
perfection in characteristic $p>0$. In this section, for a Noetherian ring $A$ of characteristic $p>0$, we
study the relationship between the Cohen--Macaulayness of its perfection $A_{\mathrm{perf}}$ and the weak
$F$-nilpotence of $A$.

\begin{defi}\label{def:F-nilpotent}
  Let $(A,\ideal{m})$ be a $d$-dimensional Noetherian local ring of characteristic $p>0$, and let $F$ denote the induced Frobenius action on each local cohomology module $H^i_{\ideal{m}}(A)$.
  We consider the Frobenius closure and tight closure of the zero submodule in local cohomology; namely, set
  \[
    0^F_{H^i_{\ideal{m}}(A)}
    :=\{\eta\in H^i_{\ideal{m}}(A)\mid F^e(\eta)=0
    \text{ for some }e\geq0\}.
  \]
  Let $A^\circ$ denote the complement of the union of the minimal primes of $A$, and set
  \[
    0^*_{H^d_{\ideal{m}}(A)}
    :=\{\eta\in H^d_{\ideal{m}}(A)\mid
    cF^e(\eta)=0\text{ for all }e\gg0
    \text{ for some }c\in A^\circ\}.
  \]
  \begin{enumerate}
    \item We say that $A$ is \emph{weakly $F$-nilpotent} if, for every $i<d$,
          \[
            H^i_{\ideal{m}}(A)=0^F_{H^i_{\ideal{m}}(A)},
          \]
          that is, if every element of $H^i_{\ideal{m}}(A)$ is annihilated by some iterate of $F$ for $i<d$.
    \item We say that $A$ is \emph{$F$-nilpotent} if $A$ is weakly $F$-nilpotent and, in addition,
          \[
            0^F_{H^d_{\ideal{m}}(A)}=0^*_{H^d_{\ideal{m}}(A)}.
          \]
          Equivalently, every element of $0^*_{H^d_{\ideal{m}}(A)}$ is annihilated by some iterate of $F$.
  \end{enumerate}
  A Noetherian ring $A$ is called \emph{locally weakly $F$-nilpotent} (resp. \emph{locally $F$-nilpotent}) if $A_P$ is weakly $F$-nilpotent (resp. $F$-nilpotent) for every $P\in\spec A$.
\end{defi}

Only weak $F$-nilpotence will be needed in this section. We first express the cohomology of the perfection
with support in terms of the Frobenius action on local cohomology.

Let $A$ be a ring of characteristic $p>0$. We write
\[
  A_{\mathrm{perf}}
  :=\varinjlim\bigl(A\xrightarrow{F}A\xrightarrow{F}\cdots\bigr)
\]
for the perfection of $A$. The following lemma relates the cohomology of $A_{\mathrm{perf}}$ to the Frobenius
action on the local cohomology of $A$.

\begin{lem}\label{lem:perfect-cohomology}
  Let $A$ be a Noetherian ring of characteristic $p>0$, let
  $\mathfrak p\in\spec A_{\mathrm{perf}}$, and set $P=\mathfrak p\cap A$.
  Then there is a natural isomorphism $(A_{\mathrm{perf}})_{\mathfrak p}\cong(A_P)_{\mathrm{perf}}$, and $\dim (A_{\mathrm{perf}})_{\mathfrak p}=\dim A_P$. Moreover, for every $i\geq0$, there is a natural isomorphism
  \[
    H^i_{\mathfrak p(A_{\mathrm{perf}})_{\mathfrak p}}
    \bigl(
    \spec (A_{\mathrm{perf}})_{\mathfrak p},
    \mathcal O
    \bigr)
    \cong
    \ilim
    \left(
    H^i_{PA_P}(A_P)
    \xrightarrow{F}
    H^i_{PA_P}(A_P)
    \xrightarrow{F}
    \cdots
    \right).
  \]
\end{lem}

\begin{proof}
  Each Frobenius map $F\colon A\to A$ induces the identity on the underlying topological space of $\spec A$, and
  \[
    \spec A_{\mathrm{perf}}
    \cong
    \plim
    \bigl(
    \spec A\xleftarrow{F}\spec A\xleftarrow{F}\cdots
    \bigr).
  \]
  Hence the natural map $\spec A_{\mathrm{perf}}\to\spec A$ is a homeomorphism. Since localization commutes
  with filtered colimits, we have
  \[
    (A_{\mathrm{perf}})_{\mathfrak p}
    \cong
    \ilim
    \bigl(
    A_P\xrightarrow{F}A_P\xrightarrow{F}\cdots
    \bigr)
    =
    (A_P)_{\mathrm{perf}}.
  \]
  Thus $\spec (A_{\mathrm{perf}})_{\mathfrak p}\to\spec A_P$ is also a homeomorphism, and therefore $\dim
    (A_{\mathrm{perf}})_{\mathfrak p}=\dim A_P$.

  Set $d=\dim A_P$, and let $\underline{a}=a_1,\dots,a_d$ be a system of parameters of $A_P$. By the above
  homeomorphism, the only prime ideal of $(A_{\mathrm{perf}})_{\mathfrak p}$ containing
  $\underline{a}(A_{\mathrm{perf}})_{\mathfrak p}$ is $\mathfrak p(A_{\mathrm{perf}})_{\mathfrak p}$. Hence
  $\sqrt{\underline{a}(A_{\mathrm{perf}})_{\mathfrak p}} = \mathfrak p(A_{\mathrm{perf}})_{\mathfrak p}$.

  By Proposition~\ref{prop:sheaf-cech} and the exactness of filtered colimits, for every $i$ we obtain
  \[
    H^i_{\mathfrak p(A_{\mathrm{perf}})_{\mathfrak p}}
    \left(
    \spec (A_{\mathrm{perf}})_{\mathfrak p},
    \mathcal O
    \right)
    \cong
    \check{H}^i\left(\underline{a},(A_{\mathrm{perf}})_{\mathfrak p}\right)
    \cong
    \ilim
    \check{H}^i\left(a_1^{p^e},\dots,a_d^{p^e},A_P\right).
  \]
  At each stage, $\check{H}^i(a_1^{p^e},\dots,a_d^{p^e},A_P) \cong H^i_{PA_P}(A_P),$ and the transition maps
  are induced by Frobenius. Thus we obtain the desired isomorphism
  \[
    H^i_{\mathfrak p(A_{\mathrm{perf}})_{\mathfrak p}}
    \bigl(
    \spec (A_{\mathrm{perf}})_{\mathfrak p},
    \mathcal O
    \bigr)
    \cong
    \ilim
    \left(
    H^i_{PA_P}(A_P)
    \xrightarrow{F}
    H^i_{PA_P}(A_P)
    \xrightarrow{F}
    \cdots
    \right).
  \]
\end{proof}

\begin{thm}\label{thm:perfect-characterization}
  Let $A$ be a Noetherian ring of characteristic $p>0$.
  The following conditions are equivalent.
  \begin{enumerate}
    \item $A$ is locally weakly $F$-nilpotent.
    \item $A_{\mathrm{perf}}$ is a locally balanced big Cohen--Macaulay $A$-algebra.
    \item $A_{\mathrm{perf}}$ is CCM\@.
  \end{enumerate}
  If these conditions hold, then $A_{\mathrm{perf}}$ is locally HMCM\@.
\end{thm}

\begin{proof}
  \textup{(1)}$\Longleftrightarrow$\textup{(2)}.
  For any $P\in\spec A$ and the unique $\mathfrak p\in\spec A_{\mathrm{perf}}$ lying over $P$, the definition of the perfection and Lemma~\ref{lem:perfect-cohomology} give
  \[
    (A_{\mathrm{perf}})_P\cong(A_{\mathrm{perf}})_{\mathfrak p}
    \cong(A_P)_{\mathrm{perf}}.
  \]
  By Ma--Polstra \cite{Ma-Polstra2025}*{Proposition~12.22}, a Noetherian local ring is weakly $F$-nilpotent if
  and only if its perfection is a balanced big Cohen--Macaulay algebra over that ring. Applying this
  characterization to every localization yields the equivalence with $A_{\mathrm{perf}}$ being a locally
  balanced big Cohen--Macaulay $A$-algebra.

  \textup{(2)}$\Longrightarrow$\textup{(3)} follows from Theorem~\ref{thm:integral-transfer}, since $A_{\mathrm{perf}}$ is integral over $A$.

  \textup{(3)}$\Longrightarrow$\textup{(1)}.
  Let $\mathfrak p\in\spec A_{\mathrm{perf}}$, set $P=\mathfrak p\cap A$, and let $d=\dim A_P$.
  For $i<d$, Lemma~\ref{lem:perfect-cohomology} and the CCM property of $(A_{\mathrm{perf}})_{\mathfrak p}$ give
  \[
    \varinjlim_e
    \left(
    H^i_{PA_P}(A_P)\xrightarrow{F}H^i_{PA_P}(A_P)\xrightarrow{F}\cdots
    \right)=0.
  \]
  The vanishing of this direct limit is equivalent to every element being annihilated by some iterate of
  Frobenius. Hence each $A_P$ is weakly $F$-nilpotent.

  Finally, $A_{\mathrm{perf}}$ is locally finite-dimensional, so under \textup{(3)} we may apply
  Theorem~\ref{thm:ccm-to-locally-hmcm}.
\end{proof}

\begin{cor}\label{cor:perfect-dimension-one}
  Let $A$ be a Noetherian ring of characteristic $p>0$ with $\dim A\leq1$.
  Then $A_{\mathrm{perf}}$ is CCM and locally HMCM\@.
\end{cor}

\begin{proof}
  Let $\mathfrak p\in\spec A_{\mathrm{perf}}$ and set $P=\mathfrak p\cap A$.
  Since $\dim A_P\leq1$, it suffices by Theorem~\ref{thm:perfect-characterization} to show that $A_P$ is weakly $F$-nilpotent.

  If $\dim A_P=0$, the condition for weak $F$-nilpotence is vacuous. Suppose that $\dim A_P=1$, and let $x\in
    H^0_{PA_P}(A_P)$. Then $(PA_P)^n x=0$ for some $n$. If $x$ were a unit, then $(PA_P)^n=0$, contradicting
  $\dim A_P=1$. Hence $x\in PA_P$, so $x^{n+1}=0$. Therefore $F^e(x)=0$ for all sufficiently large $e$, and
  $A_P$ is weakly $F$-nilpotent.

  By Lemma~\ref{lem:perfect-cohomology}, the natural map $\spec A_{\mathrm{perf}}\to\spec A$ is a
  homeomorphism, so $P$ ranges over all prime ideals of $A$. Thus $A$ is locally weakly $F$-nilpotent.
  Theorem~\ref{thm:perfect-characterization} now shows that $A_{\mathrm{perf}}$ is CCM and locally HMCM\@.
\end{proof}

\subsection*{Use of AI}
The author used ChatGPT (OpenAI) for English translation, language editing, proofreading, and checking the
exposition for possible inconsistencies. The author carefully reviewed and manually modified the AI-generated
outputs, and takes full responsibility for the final contents.

\end{document}